\documentclass[12pt]{amsart}

\usepackage{amsfonts,amsmath,latexsym,amssymb,verbatim,amsbsy,times}
\usepackage{amsthm}
\usepackage{pstricks}
\usepackage{graphicx}
\usepackage{caption}
\usepackage[top=1in, bottom=1in, left=0.8in, right=0.8in]{geometry}

\usepackage[colorlinks=true, pdfstartview=FitV, linkcolor=blue,citecolor=green, urlcolor=blue]{hyperref}

\usepackage{enumitem}

\theoremstyle{plain}
\newtheorem{THEOREM}{Theorem}[section]

\newtheorem{LEMMA}[THEOREM]{Lemma}

\newtheorem{theorem}[THEOREM]{Theorem}

\newtheorem{lemma}[THEOREM]{Lemma}

\theoremstyle{definition}

\newtheorem{definition}[THEOREM]{Definition}

\theoremstyle{remark}
\newtheorem{remark}[THEOREM]{Remark}
\newtheorem{REMARK}[THEOREM]{Remark}

\newcommand{\thm}[1]{Theorem~\ref{#1}}
\newcommand{\lem}[1]{Lemma~\ref{#1}}

\newcommand{\Z}{\ensuremath{\mathbb{Z}}}   
\newcommand{\R}{\ensuremath{\mathbb{R}}}   
\newcommand{\T}{\ensuremath{\mathbb{T}}}   

\def \a {\alpha}
\def \b {\beta}
\def \d {\delta}
\def \g {\gamma}
\def \e {\epsilon}
\def \f {\varphi}

\def \l {\lambda}
\def \L {\Lambda}

\def \n {\nabla}
\def \s {\sigma}

\def\cprime{$'$}

\def \< {\langle}
\def \> {\rangle}
\def \p {\partial}
\def \ra {\rightarrow}
\def \ss {\subset}

\DeclareMathOperator{\supp}{supp} %

\def \cH {\mathcal{H}}

\def \dom {\R^3}

\begin{document}

\title{Conditions Implying Energy Equality for Weak Solutions of the Navier--Stokes Equations}
\author{Trevor M. Leslie and Roman Shvydkoy}

\email{tlesli2@uic.edu; shvydkoy@uic.edu}

\address{Department of Mathematics, Statistics, and Computer Science \\851 S Morgan St, M/C 249 \\ University of Illinois at Chicago, Chicago, IL 60607}
 
 \thanks{This work was partially supported by NSF grants DMS-1210896 and DMS-1515705. The authors thank V. {\v{S}}ver{\'a}k for valuable discussions and constant interest in the work. We also thank the anonymous referees for numerous helpful comments and suggestions.}


\maketitle
\begin{abstract}
When a Leray--Hopf weak solution to the NSE has a singularity set $S$ of dimension $d$ less than~$3$---for example, a suitable weak solution---we find a family of new $L^q L^p$ conditions that guarantee validity of the energy equality. Our conditions surpass the classical Lions--Lady{\v{z}}enskaja $L^4 L^4$ result in the case $d<1$. Additionally, we establish energy equality in certain cases of Type-I blowup.  The results are also extended to the NSE with fractional power of the Laplacian below $1$.
\end{abstract}

\section{Introduction}

Consider the incompressible Navier--Stokes equations
\begin{equation}
\label{e:momentum}
\p_t u + u\cdot \n u - \nu\Delta u = - \n p
\end{equation}
\begin{equation}
\label{e:divfree}
\n \cdot u = 0
\end{equation}
where $u$ is the velocity field, $p$ is the scalar pressure, and $\nu$ is the viscosity.  We restrict attention to the case of the open domain $\R^3$ for definiteness. The results below carry over ad verbatim to $\T^3$ and locally to the interior of a bounded domain as well.

By a classical result of Leray \cite{Leray}, it is known that for divergence-free initial data $u_0\in L^2$, there exists a weak solution to \eqref{e:momentum}--\eqref{e:divfree} up to a specified time $T$ such that $u\in L^2 H^1\cap L^\infty L^2$ and 
\begin{equation}
\label{e:global_ei}
\int_{\R^3\times \{t\}} |u|^2\,dx 
\le  \int_{\R^3\times \{t_0\}} |u|^2\,dx -2\nu \int_{t_0}^t \int_{\R^3} |\n u |^2\,dx\,dt
\end{equation}
for all $t\in (0,T]$ and a.e. $t_0\in [0,t]$ including $t_0 = 0$. 
Moreover, strong solutions to \eqref{e:momentum}--\eqref{e:divfree} satisfy the corresponding energy equality:
\begin{equation}
\label{e:global_ee}
\int_{\R^3\times \{t\}} |u|^2\,dx - \int_{\R^3\times \{0\}} |u|^2\,dx
= -2\nu \int_{t_0}^t \int_{\R^3} |\n u |^2\,dx\,dt.
\end{equation}
Since the introduction of Leray--Hopf solutions, it has been notoriously difficult to establish energy equality for all such solutions. The question, beyond purely mathematical interest, is motivated on physical grounds as well: Knowing \eqref{e:global_ee} rather than \eqref{e:global_ei} rules out the presence of anomalous energy dissipation due to the nonlinearity, a phenomenon normally associated with weak solutions of the inviscid Euler system in the framework of the so-called Onsager conjecture \cite{Onsager} (more on this below).  This allows, as stipulated, for example, in the text of Frisch \cite{Frisch}, to precisely equate the classical Kolmogorov residual energy anomaly $\e_\nu \ra \e_0$ of a turbulent flow to the Onsager dissipation in the limit of vanishing viscosity.

Let us give a brief overview of what has been done so far in the direction of resolving the question of energy equality.   Lions proved \cite{Lions} that \eqref{e:global_ee} holds for $u\in L^4 L^4$; techniques developed in the classical book of Lady{\v{z}}enskaja, Solonnikov, and Ural{\cprime}ceva \cite{Ladyzhenskaya} reproduce this result. Later, Serrin \cite{Serrin} proved energy equality in space dimension $n$ under the condition $ \frac np + \frac2q \le 1$.  Shinbrot \cite{Shinbrot} improved upon this result, proving equality when $\frac 2p + \frac2q \le 1$, $ p\ge 4$, independent of the dimension.  Kukavica \cite{Kukavica} has proven equality under the assumption $p\in L^2 L^2$; this assumption is weaker than---but dimensionally equivalent to---Lions's result.  A number of new conditions have appeared more recently after the introduction of critical conditions for the parallel question of energy conservation for the Euler system (cf. \cite{CET}, \cite{DR}, \cite{CCFS}). In \cite{CCFS}, energy equality is shown to hold for both the Euler and the Navier--Stokes systems for all solutions in the Besov-type regularity class
\def \reg {\mathcal{R}}
\begin{equation}\label{e:Besov}
\reg_0 = \left\{ u \in  L^3_t L^3_x : \lim_{|y|\ra 0 } \frac{1}{|y|} \int_{\R^n\times[0,T]}|u(x+y,t) - u(x,t)|^3\, dx \, dt  = 0 \right\}.
\end{equation}
Note that this class measures regularity $1/3$ in space, ``$L^3$-averaged" over space-time. In particular, the condition defining $\reg_0$ holds if  $u\in L^3 B^{1/3}_{3,p}$ for some $p\in [1,\infty)$; the class includes spaces like $L^3W^{1/3, 3}$ and $L^3 H^{5/6}$. On a bounded domain, the energy equality is established in \cite{CFS} for the dimensionally equivalent class $L^3 D(A^{5/12})$, where $A$ is the Stokes operator; see also \cite{FT} for extension to exterior domains. Let us note that by interpolation with the enstrophy class $L^2 H^1$, any solution in $L^4 L^4$ lands in $L^3 B^{1/3}_{3,3} \ss \reg_0$. Thus, Lions's condition can be recovered from Onsager's.  

It was not until after most of the results above had been proven that arguments establishing \eqref{e:global_ee} began to make use of the fact that the set of singular points of a weak solution  may be confined to a lower-dimensional subset of time-space. This is of course the case for suitable weak solutions, according to the Caffarelli--Kohn--Nirenberg (CKN) theorem \cite{CKN}. In \cite{SS-pressure}, the authors examine the situation where $u$ is bounded in an energy class which is scaling invariant in space, and the energy equality is established by covering the singularity set in accordance with the CKN theorem. Presently, we can address these cases in a systematic way with the use of the class $\reg_0$. Indeed, any condition on the solution $u$ which is spatially both shift-invariant and scale-invariant implies that $u$ belongs to $L^\infty B^{-1}_{\infty,\infty}$, the largest such class by Cannone's theorem \cite{Cannone}. By interpolation with $L^2 H^1 = L^2 B^{1}_{2,2}$, we find again that $u \in L^3 B^{1/3}_{3,3} \ss \reg_0$, and \eqref{e:global_ee} follows (see Section~\ref{ss:typeI}). A cutoff procedure was also previously used in \cite{ShvydkoyGeometric} to establish energy equality; there it was assumed that the singularity was confined to a curve $s \in C^{1/2}([0,T]; \R^3)$ and additionally that $u\in L^3 L^{9/2}$, $\n u \in L^3 L^{9/5}((0,T)\times \R^3\backslash\mathrm{Graph}(s))_{\mathrm{loc}}$, the assumption dimensionally equivalent to the class $\reg_0$.  

In this paper, we establish new sufficient conditions for energy equality which specifically exploit low dimensionality of the singularity set. We consider both classical and fractional dissipation cases. The results are sorted into two categories: the more special case where \eqref{e:global_ee} is established on a time interval of regularity until the first time of blowup, and the more general case of singularities spread over space-time. In the former case the results are stronger. Although it is more restrictive in terms of setup, it is also the case that is most relevant for the blowup problem. The conditions we find depend on the dimension $d<3$ of the singularity set, which is defined precisely below. The bifurcation value of $d$ occurs at $d=1$, or $d=5-4\g$ in the fractional case, where $\g$ is the power of the Laplacian. Recall that if $u$ is a suitable solution of the classical NSE, then by CKN we have $d\le 1$, so the low dimensionality comes as given in this case.  

We state our main result in terms of suitable solutions to the classical Navier--Stokes equation, as it appears to be the most addressed case in the literature. However, this result is a special case of a much more general set of criteria depending on values of $d$ and $\g \leq 1$, which we will state in detail in the sections below. To illustrate our results, we make extensive use of diagrams, drawn in $(x = 1/p,\,y=1/q)$ coordinates.  The striped regions in our figures correspond to new values of $p$ and $q$ for which the condition $u\in L^q L^p$ implies energy equality.  A dotted boundary indicates that values on the boundary are not included, while a solid line indicates included values.

\begin{theorem} \label{t:main}
Suppose $u\in C_w([0,T];L^2) \cap L^2([0,T]; H^1)$ is a suitable weak solution on $[0,T]$, regular on $[0,T)$. Assume that $u \in L^q L^p$, where one of the following conditions holds (see Figure~\ref{fig:d1}):
\begin{align}
\frac{2}{p} + \frac{2}{q} & \le 1, \;\; 3 \leq q \le p   \label{t:opt1} \\
\frac{2}{p} + \frac{2}{q} &< 1, \;\; 3 \leq p < q \label{t:opt2}\\
\frac{7}{p}-\frac{6}{p^2} + \frac{2}{q} & < 2, \;\; p<3.
\end{align}
Then $u$ satisfies \eqref{e:global_ee} on the interval $[0,T]$. 
\end{theorem}

Theorem \ref{t:main} will be proven in Section~\ref{s:1-slice} as part of a more general result for dimensions $d<3$ on an interval of regularity. The results are summarized in Figures~\ref{fig:dzero}, \ref{fig:done}, \ref{fig:d1}, \ref{fig:d13}.  We can also treat the situation where $u$ has one of the following Type-I blowups at $T$: 
\begin{equation}\label{e:typeI}
\sup_x |u(x,t)| \leq \frac{C}{\sqrt{T-t}} \text{ or } \sup_{0<t<T} |u(x,t)| \leq \frac{C}{|x|}.
\end{equation}
We call these two scenarios ``Type-I in time'' blowup and ``Type-I in space'' blowup, respectively.

\begin{theorem}\label{t:typeI}
Suppose $u$ is a Leray--Hopf solution on $[0,T]$ which is regular on $[0,T)$.  If $u$ experiences Type-I in space blowup as in \eqref{e:typeI}, then $u$ satisfies \eqref{e:global_ee} on $[0,T]$.  If $u$ experiences Type-I in time blowup as in \eqref{e:typeI} and additionally $d<1$, where $d$ denotes the Hausdorff dimension of the singularity set at time $T$, then $u$ satisfies \eqref{e:global_ee} on $[0,T]$.
\end{theorem}

General singularity sets which are spread out in space-time will be addressed in Section~\ref{s:gen} for the classical NSE; the results are depicted in Figure~\ref{fig:0<d<1_gen} for $0<d<1$. At $d=1$, the new region collapses to the known classical diagram; see Figure~\ref{fig:d=1_gen}. We give extensions for the fractional dissipation case in Section~\ref{s:frac} and present similar figures for each significantly distinct range of values $d,\g$ pertaining to the time-slice singularity case. On the way, we prove a commutator estimate in Lemma~\ref{p:commut} which may be of independent interest.

\section{Setup} 
As our first order of business, we make precise the notion of the regular and singular sets under consideration in our analysis. We follow the setup of \cite{Shvydsing}.  First, we define two regularity classes of vector fields, reminiscent of \eqref{e:Besov}.  For a subinterval $I\subset [0,T]$, we denote
\begin{equation}
\reg(\R^n \times I) 
= \left\{ u \in  L^3_t L^3_x : \lim_{|y|\ra 0 } \frac{1}{|y|} \int_{\R^n\times I}|u(x+y,t) - u(x,t)|^3 \, dx \, dt  = 0 \right\}.
\end{equation}
We also define a local version of this class, denoting by $\reg(U\times I)$ the class of vector fields $u$ such that $u\phi\in \reg(\R^n\times I)$ for all $\phi\in C_0^\infty(U)$, where $U$ is any open set in $\R^n$. 

\begin{definition}
Let $u$ be a Leray--Hopf weak solution to the classical Navier--Stokes equations on $\R^3\times [0,T]$.  We say that a point $(x_0,t_0)$ is an \emph{Onsager regular point} if $u\in \reg(U\times I)$ for some open set $U\subset \R^3$ and relatively open interval $I\subset [0,T]$ such that $(x_0, t_0)\in U\times I$.  We say that $(x_0, t_0)$ is an \emph{Onsager singular point} if it is not a regular point; further, we denote the (closed) set of all Onsager singular points by $\Sigma_{ons}$ and refer to it as the \emph{Onsager singular set}.  The complement of $\Sigma_{ons}$ in $\R^3\times[0,T]$ is called the \emph{Onsager regular set} of $u$.
\end{definition}

\begin{REMARK}
The CKN theorem implicates a different type of singularity set which we denote $\Sigma_{CKN}$.  This set can be defined as the complement in $\R^3\times [0,T]$ of 
\begin{equation}
\reg_{CKN} = \{(x_0,t_0)\in \R^3\times [0,T] : \exists \text{ nbd. } D\subset \R^3\times [0,T] \text{ of } (x_0,t_0) \text{ s.t. } u\in L^\infty(D) \},
\end{equation}
i.e., $\Sigma_{CKN} = (\R^3\times [0,T])\backslash\reg_{CKN}$.  Clearly $\Sigma_{ons}\subset \Sigma_{CKN}$ so that in particular the bounds on the size of $\Sigma_{CKN}$ from the CKN theorem apply a fortiori to $\Sigma_{ons}$.  
\end{REMARK}

Our next item is to introduce a local energy equality which is fundamental to our work.  Suppose $(u,p)$ is a Leray--Hopf weak solution to the Navier--Stokes system on $\dom \times [0,T]$, and consider the following local energy equality for $0\le s<t\le T$:
\begin{multline}\label{localee}
\int_{\dom} |u(t)|^2 \phi - \int_{\dom} |u(s)|^2 \phi - \int_{\dom \times (s,t)} |u|^2 \p_t \phi \\= \int_{\dom \times (s,t)} |u|^2  u \cdot \n \phi + 2 \int_{\dom \times (s,t)} p\ u \cdot \n \phi - 2\nu \int_{\dom \times (s,t)} |\n u|^2 \phi - 2 \nu \int_{\dom \times (s,t)} u \otimes \n \phi : \n u.
\end{multline}
The main idea of the present work is to construct a sequence of test functions which satisfy this equality and to show that when we pass to the limit, the local energy equality reduces to \eqref{e:global_ee}.  It is shown in \cite{Shvydsing} that \eqref{localee} is valid for all $\phi\in C_0^\infty((\R^3\times [0,T])\backslash \Sigma_{ons})$ in the case of the Euler equations ($\nu = 0$).  Straightforward modifications of the proof in \cite{Shvydsing} show that \eqref{localee} is also valid when $\nu>0$.  In fact, an approximation argument shows that \eqref{localee} remains valid for functions $\phi$ (supported outside $\Sigma_{ons}$, as before) which belong only to $W^{1,\infty}$ rather than $C^\infty$. 

Recall that Leray--Hopf solutions satisfy $u(t)\to u(0)$ strongly in $L^2(\R^3)$ as $t\to 0^+$.  Therefore, in order to establish \eqref{e:global_ee}, it suffices to prove energy balance on the time interval $[s,T]$ for each $s\in (0,T)$; the (Onsager) singularity set at the initial time is irrelevant for our analysis. Therefore, we introduce the following singularity set, which we call the \emph{postinitial singularity set} $S$ (or simply the \emph{singularity set} when it will cause no confusion), defined by 
\[
S = \Sigma_{ons}\backslash(\R^3\times \{0\}).
\]
Working with $S$ rather than all of $\Sigma_{ons}$ allows us to obtain better conditions guaranteeing energy balance for solutions which have arbitrary divergence free initial condition $u_0\in L^2$ (but which have small postinitial singularity sets).  A priori, this replacement requires us to assume $s>0$ rather than $s\ge 0$ in \eqref{localee}.  However, as pointed out above, we may extend to $s=0$ by continuity, so that we may consider $S$ instead of $\Sigma_{ons}$ at no real cost. We will make the standing assumption that the Lebesgue measure $|S|$ of $S$ in $\R^3\times [0,T]$ is equal to zero. 

Let us label each of the terms in \eqref{localee} (in the same order as before) and rewrite the equation as
\begin{equation}
A - B -C = D +2 P - 2\nu E - 2 \nu F.
\end{equation}

Having established the above considerations and notation, we can now describe the main idea more clearly and succinctly.  Given a Leray--Hopf solution $u$ and its (postinitial) singularity set $S$, we seek a sequence $\{\phi_\d\}_{\d>0}$ of test functions such that 
\begin{itemize}
	\item $\supp \phi_\d \subset (\R^3\times[0,T])\backslash S$ and $\phi_\d\in W^{1,\infty}(\R^3\times [0,T])$ (so \eqref{localee} is valid for all $0<s<t \le T$);
	\item $0\le \phi_\d\le 1$ and $\phi_\d\to 1$ pointwise a.e. as $\d\to 0$ (which is possible since $|S|=0$), guaranteeing the convergence of the terms $A$, $B$, and $E$ to their natural limits
	\[
	\int_{\dom} |u(t)|^2, \int_{\dom} |u(s)|^2, \int_{\dom \times (s,t)} |\n u|^2,
	\] 
	respectively.  These convergences follow from the fact that $u\in L^\infty L^2\cap L^2 H^1$, together with the dominated convergence theorem.
\end{itemize}
When $A$, $B$, and $E$ tend to their natural limits, we see that in order to establish energy balance on $[s,T]$, it suffices to prove that the other terms $C$, $D+2P$, and $F$ vanish as $\d\to 0$.  In order to ensure this, we make integrability assumptions on the solution $u$, i.e., $u\in L^q([0,T],L^p(\R^3))$ for some pair $(p,q)$ of integrability exponents.  The set of admissible values for $p$ and $q$, which will make the terms $C$, $D+2P$ and $F$ vanish, depend on the integrability properties of the functions $\phi_\d$, which in turn depend on the size and structure of $S$.  Therefore, we continue our discussion in the sections below, where we restrict attention to certain kinds of singularity sets $S$.  Note that in the discussion below, we generally suppress the notation $\d$ from the subscript of our sequence of test functions.

\section{Energy equality at the first time of blowup}\label{s:1-slice}

The case addressed in this section pertains to the situation when singularity $S$ occurs only at the critical time $T$. For notational convenience, we will replace the interval $[0,T]$ with $[-1,0]$, $0$ being critical, and thus assume that $S\subset \R^3\times \{0\}$.

\subsection{Construction of the test function}
We assume that $S$ has Hausdorff dimension $d<3$. (Recall that if $(u,p)$ is a suitable solution, then by CKN, we have $d\leq 1$.) For convenience, we will identify $S$ with its spatial slice at time $t=0$. We denote by $\cH_d(S)$ the $d$-dimensional Hausdorff measure of $S$ and assume that $\cH_d(S)<\infty$. In  what follows below, we take advantage of the fact that $S$ belongs only to the time-slice at $t=0$ and that we can therefore cover $S$ with cylinders scaled arbitrarily in time. Specifically, let us denote by $B_r(x)$ the open ball $\{y\in \R^3: |y-x| < r\}$. Choose $\d\in (0,1)$, then choose finitely many $x_i \in \R^3$, $r_i\in (0,\d)$ for all $i$, such that $S\subset \bigcup_i B_{r_i}(x_i)$ and $\sum_{i=1}^\infty r^d_i \le \cH_d(S)+1$.  Denote $I_i = (-2r_i^{\a},\,2r_i^{\a})$ (where $\a$ is determined below); let $Q_i$ denote the cylinder $Q_i = B_{r_i}(x_i)\times (-r_i^{\a},\,r_i^{\a})$, and put $Q = \bigcup_i Q_i$, $I = \bigcup_i I_i$.  Let $\psi(s)$ be the usual (symmetric, radially decreasing) cutoff function on the line with $\psi(s) = 1$  on $|s|<1.1$ and $\psi(s)$ vanishing on $|s|>1.9$. Let $\phi_i(x,t) = \psi(|x-x_i|/r_i) \psi(t/r_i^{\a})$. Define $\phi = 1 - \sup_i \phi_i$. Clearly, $\phi$ vanishes on an open neighborhood of $Q$, while any partial derivative $\p \phi$ is supported within the union of the double-dilated cylinders, which we denote by $Q^*$.  Note that the Lebesgue measure of the sequence of $Q^*$'s vanishes as $\d \ra 0$; the same is true of the measure of the sequence of $I$'s. Also note that $\phi$ is differentiable a.e.\ and 
\[
|\p \phi(x,t) | \leq \sup_i |\p \phi_i(x,t)| \, \text{ a.e.; see \cite[Theorem 4.13]{Evans}}.
\]
Therefore, for any $a>0$, we have the following bounds, which hold for a.e. $t$:
\begin{subequations} \label{e:phi}
	\begin{align}
	\int_{\dom} |\p_t \phi(x,t)|^a\,dx & \leq \sum_i \int_{\dom} |\p_t \phi_i (x,t)|^a\,dx \leq \sum_i r_i^{-\a  a + 3}\chi_{I_i}(t) \label{e:phit} \\
	\int_{\dom} |\n_x \phi(x,t)|^a\,dx &\leq \sum_i \int_{\dom} |\n_x \phi_i(x,t)|^a\,dx \leq \sum_i r_i^{-a+3}\chi_{I_i}(t). \label{e:phix}
	\end{align}
\end{subequations}

\subsection{Type-I singularities}\label{ss:typeI}
Generally we say that a solution $u$ of the classical Navier--Stokes equations experiences a Type-I blowup at time $0$ if it stays bounded in some scale-invariant functional space:
\[
\| u\|_{X([-1,0]; Y)} \leq C.
\]
Examples include those stated in \eqref{e:typeI}. It also occurs naturally in the case of a self-similar blowup with critical decay of the profile at infinity,
\[
u(x,t) = |t|^{-1/2} U(x / |t|^{1/2}),\quad |U(y)| \leq C / |y|,  \text{ as } |y| \to \infty.
\]
In this case, $u$ clearly belongs to the Lions space $L^4 L^4$ and therefore satisfies the energy equality.  A more subtle situation occurs in the case of Type-I in space only or Type-I in time only blowup, which is addressed in our \thm{t:typeI}.  By Type-I in space, we mean a weak solution on a time interval $[-1,0]$ with the bound given by the second inequality in \eqref{e:typeI} (technically, in this case,  multiple blowups are possible on the interval). 

Now, any solution $u$ on the time interval $[-1,0]$ which experiences Type-I in space blowup belongs to the class $L^\infty(-1,0,L^{3,\infty}(\R^3))$.  It can be seen in (at least) two different ways that solutions in this class necessarily satisfy the energy balance relation.  First, we see that $L^2 L^6 \cap L^\infty L^{3,\infty}\subset L^4 L^4$, simply by interpolation, so that the Lions criterion can be used. Alternatively, we can apply Cannone's theorem \cite{Cannone} to the space $L^{3,\infty}$, which is invariant with respect to both shifts $f\mapsto f(\cdot - x_0)$ and rescalings of the form $f \mapsto \l f(\l \cdot)$, allowing us to conclude that $L^{3,\infty}$ embeds in the largest space with these properties, namely, $B^{-1}_{\infty, \infty}$. By interpolation with $L^2 H^1 = L^2 B^{1}_{2,2}$, we naturally find $u \in L^3 B^{1/3}_{3,3} \ss \reg_0$, which implies energy equality as mentioned in the introduction. This settles the first part of \thm{t:typeI}.  

By Type-I in time, we mean a regular solution $u$ on time interval $[-1,0)$ that experiences blowup at time $t=0$ and satisfies the first inequality in \eqref{e:typeI}. If $u$ is a Type-I in time solution, then $u\in L^r L^\infty$ for any $r<2$.  If additionally we have that $0\le d<1$, then we can choose $r<2$ large enough so that the pair $(p,q) = (\infty, r)$ satisfies \eqref{opt1} below.  We will see that this is a sufficient condition to guarantee \eqref{e:global_ee}. This resolves the second claim in \thm{t:typeI}.

\subsection{Vanishing of the terms $C,D,P,F$ in the one-time singularity case.} We now turn to the proof of \thm{t:main}, which encompasses the next two subsections.  Actually, we will address the time-slice singularity case whenever $S$ has Hausdorff dimension $d<3$, giving a range of $L^q L^p$ conditions for which energy equality holds.  Of course, the case $d=1$ is the one which is relevant for purposes of Theorem \ref{t:main}. 

The outline of our argument is as follows: We will start with basic estimates on the terms $C$, $D$, $P$, and $F$, the terms in the local energy equality that depend on derivatives of $\phi$ (and hence are singular).  In this subsection, we give conditions on $p$, $q$, $d$, and $\a$ that guarantee that each of the terms $C$, $D+2P$, and $F$ vanish as $\d\to 0$; as argued above, energy balance is achieved when all of these vanish concurrently.  We treat $d$ as fixed; therefore, for each value of $\a>0$, we get a different collection of pairs $(p,q)$ for which $u\in L^q L^p$ implies energy balance.  In the following subsection, we take the union over $\a$ of all such regions to obtain all possible pairs $(p,q)$ for which our method is valid.  However, in order to record our results as explicitly as possible, we frame the process of taking this union as an optimization problem; see below.  Once this optimization problem has been solved, there is nothing more to prove, and we conclude our discussion of the one-time singularity at that point.  

Let us bound term $C$ first. By H\"older's inequality, we have that for all $p,q \geq 2$,
\begin{equation}\label{e:C}
\begin{split}
|C|& \leq \|u\|_{L^q(I;L^p)}^2 \left( \int_{-1}^0 \left( \int_{\dom} |\p_t \phi(x,t)|^{\frac{p}{p-2}}\,dx \right)^{\frac{p-2}{p}\frac{q}{q-2}} dt \right)^{\frac{q-2}{q}} \\
& \leq \|u\|_{L^q(I;L^p)}^2 \left( \int_{-1}^0 \left( \sum_i r_i^{- \frac{\a p}{p-2} + 3}\chi_{I_i}(t) \right)^{\frac{p-2}{p}\frac{q}{q-2}} dt \right)^{\frac{q-2}{q}}.
\end{split}
\end{equation}
Note that if $q<\infty$, then we have $\|u\|_{L^q(I,L^p)} \ra 0$ since $|I|\to 0$ as $\d\to 0$. So in the case $q<\infty$, in order to conclude that $C\to 0$, it suffices to prove that the term in parentheses is bounded as $\d\to 0$; the latter need not vanish. Vanishing of this term (as well as $D+2P$ and $F$; see below) for certain pairs $(p,\infty)$ will follow by interpolation.

The viscous term $F$ is bounded by
\begin{equation}\label{}
|F| \leq  \int_{Q^*} |u|^2 |\n \phi|^2\,dx\,dt + \int_{Q^*} |\n u|^2\,dx\,dt.
\end{equation}
Clearly, the second integral on the right vanishes as $\d \ra 0$. For the first integral, we have a bound similar to $C$:
\begin{equation}\label{e:F}
 \int_{Q^*} |u|^2 |\n \phi|^2\,dx\,dt \leq \|u\|_{L^q(I;L^p)}^2  \left( \int_{-1}^0 \left( \sum_i r_i^{-\frac{2p}{p-2} + 3} \chi_{I_i}(t) \right)^{\frac{p-2}{p} \frac{q}{q-2}}\,dt \right)^{\frac{q-2}{q}}.
\end{equation}
Before we proceed with estimates for $D$ and $P$, let us produce conditions on $p$ and $q$ that guarantee vanishing of the right-hand sides of \eqref{e:C} and \eqref{e:F}. The following lemma will assist us.

\begin{LEMMA}\label{l:conv} Let $d$, $\d$, $r_i$, $I_i$ be as above, and let $\s,s$ be positive numbers.  Suppose the sum $H = \sum_i r_i^d$ is finite. Then the inequality 
\begin{equation}
\label{lemma_ineq}
\int \left( \sum_i r_i^{-\s}\chi_{I_i}(t) \right)^s\,dt  \lesssim H^s
\end{equation}
holds whenever $s\ge 1$ and $s(\s + d)\le \a$ or $s<1$ and $s(\s + d)< \a$; the implied constant is independent of $\d$.  When $d = 0$, the above holds (trivially) under the nonstrict assumption $s\s \leq \a$.
\end{LEMMA}
\begin{proof}
Case 1. $s\ge 1$. By H\"older's inequality, we have 
\begin{equation}\label{}
\begin{split}
\left( \sum_i r_i^{-\s}\chi_{I_i}(t) \right)^s
 =  \left( \sum_i r^d_i r_i^{-\s-d}\chi_{I_i}(t) \right)^s 
& \leq \left( \sum_i r_i^{d} \right)^{s-1} \sum_i r_i^{d-(\s+d)s}\chi_{I_i}(t) \\
& = H^{s-1} \sum_i r_i^{d-(\s+d)s}\chi_{I_i}(t) .
\end{split}
\end{equation}
Integrating in time, we obtain 
\[
\int \left( \sum_i r_i^{-\s}\chi_{I_i}(t) \right)^s\,dt \lesssim  H^{s-1} \sum_i r_i^{d-(\s+d)s+\a}.
\]
The sum is at most $H$ whenever the condition stated in the lemma is satisfied.

Case 2. $s<1$. For each $j\in \Z$, define $R_j:=\{r_i: r_i \in [2^{-j},\,2^{-j+1})\}$, and let $N_j$ denote the cardinality of $R_j$. Clearly, $N_j\lesssim 2^{jd} H$ and $N_j=0$ for $j\le 0$. Also denote $J_j = [-2^{(-j+1)\a},\,2^{(-j+1)\a}]$.  So if $r_i \in R_j$, then $r_i^{-\s}\chi_{I_i}(t)\le 2^{j\s}\chi_{J_j}(t)$. Therefore,
\[
\begin{split}
\int \left( \sum_i r_i^{-\s}\chi_{I_i}(t) \right)^s\,dt & \le \int \left( \sum_j N_j 2^{j\s}\chi_{J_j}(t)\right)^s \,dt
\lesssim H^s \int \left( \sum_j 2^{j(\s+d)}\chi_{J_j}(t)\right)^s \,dt \\
& \leq H^s \int \sum_{j=1}^\infty 2^{j(\s+d)s}\chi_{J_j}(t)\,dt \lesssim   H^s \sum_{j=1}^\infty 2^{j((\s+d)s - \a)}.
\end{split}
\]
The final sum converges to an adimensional number by the assumption of the lemma.
\end{proof}

With \lem{l:conv} in hand, we continue our discussion of the terms $C$ and $F$ for various values of $p,q \geq 2$ and arbitrary $\a>0$. In order to obtain the desired conditions which guarantee vanishing of these terms, it suffices to translate between the quantities $s, \s$ in the lemma and the integrability exponents $p$ and $q$ at hand.

First, note that $p<q \iff \frac{p-2}{p}\frac{q}{q-2}<1$ and $p\geq q \iff \frac{p-2}{p}\frac{q}{q-2}\geq 1$. Applying \lem{l:conv}  with $\s = \frac{\a p}{p-2} - 3$ and $s = \frac{p-2}{p}\frac{q}{q-2}$, we see that $C$ vanishes whenever 
\begin{equation}
\label{C}
\frac{3-d}{p} + \frac{\a}{q}\le \frac{3-d}{2},\; p\ge q \geq 2;\quad 
\frac{3-d}{p} + \frac{\a}{q}< \frac{3-d}{2},\; 2\leq p< q.
\end{equation}
Reasoning similarly, we have $F\to 0$ whenever
\begin{equation}
\label{F}
\frac{3-d}{p} + \frac{\a}{q}\le \frac{3-d + \a-2}{2},\; p\ge q \geq 2;\quad 
\frac{3-d}{p} + \frac{\a}{q} < \frac{3-d + \a-2}{2},\;2 \leq p< q.
\end{equation}
In both cases, the conditions are nonstrict if $d=0$. 

We now turn our attention to the terms $D$ and $P$.  The estimates we use to bound these terms depend on whether $p\ge 3$ or $p<3$; we consider each case in turn.  First, suppose $p,q\in [3,\infty)$. Using H\"older's inequality together with the bound $\|up\|_{L^{q/3}(I;L^{p/3})}\lesssim \|u\|_{L^q(I;L^p)}^3$, we have the following bound:
\begin{equation}
\label{}
|D+2P| \lesssim \|u\|_{L^qL^p}^3  \left(  \int_{-1}^0 \left( \sum_i   
r_i^{-\frac{p}{p-3} +3} \chi_{I_i}(t) \right)^{\frac{p-3}{p} \frac{q}{q-3}} dt \right)^{\frac{q-3}{q}}.
\end{equation}
Arguing as before with the use of Lemma~\ref{l:conv}, we see that $|D+2P|\to 0$ whenever 
\begin{equation}
\label{DP}
\frac{3-d}{p} + \frac{\a}{q}\le \frac{2 + \a - d}{3},\; 3\le q\le p<\infty;\quad 
\frac{3-d}{p} + \frac{\a}{q}< \frac{2 + \a - d}{3},\; 3\leq p< q<\infty,
\end{equation}
with the nonstrict inequality in both cases if $d=0$.

In the case $p<3$, we can no longer use H\"older's inequality alone to bound the term $D+2P$. Instead, we will use interpolation with the enstrophy norm. When $p<3$, we have
\begin{equation}
\label{e:DPpl3}
|D+2P|\lesssim \|u\|_{L^2 H^1}^{3\b} \|u\|_{L^q L^p}^{3(1-\b)}\|\n \phi\|_{L^\s L^\infty}, 
\end{equation}
where 
\begin{equation}
\label{e:defbeta}
\frac13 = \frac{\b}{6} + \frac{1-\b}{p}\implies 
\b = \frac{6-2p}{6-p};
\quad \frac1\s = 1 - \frac{3\b}{2} - \frac{3(1-\b)}{q} = \frac{2pq - 3p - 3q}{(6-p)q}.
\end{equation}
Now 
\[
\|\n \phi\|_{L^\s L^\infty}^\s = \int \sup_i r_i^{-\s} \chi_{I_i}(t)\,dt \le \int \sup_j 2^{j\s} \chi_{J_j}(t)\,dt \lesssim \sum_j (2^{\a-\s})^{-j},
\]
and the sum on the right is bounded whenever $\s < \a$.  Substituting in for $\s$ and simplifying, we obtain
\begin{equation}
\label{DPenstr}
\frac{2 + \a}{p} + \frac{\a}{q} < \frac{1 + 2\a}{3},  \quad p<3 \leq q.
\end{equation}

\subsection{Optimization and the main result} Let us now discuss the optimal values of $\a$, beginning with the case $p\ge 3$.  Here we have the three constraints \eqref{C}, \eqref{F}, and \eqref{DP}, representing a triple of parallel lines. The $C$-line pivots around the energy space $L^\infty L^2$ and rotates counterclockwise (toward a more stringent condition) as $\a$ increases. The $DP$-line pivots around $L^3 L^\frac{9-3d}{2-d}$ and also rotates counterclockwise (but toward a more relaxed condition) as $\a$ increases. The $F$-line pivots around $L^2 L^\frac{6-2d}{1-d}$ counterclockwise, relaxing as $\a$ increases. (Note that some exponents can become negative for larger $d$'s; however, the region beyond $p,q=\infty$ can be disregarded at this moment.) Therefore, the conditions become optimal when the two lower lines coincide. Simple linear algebra shows that for $d<1$, the $C$- and $DP$-lines are lower; for $d>1$, the $C$- and $F$-lines are lower; and at $d=1$, all three lines coincide at their optimal tilt. So in the case $d\leq 1$, we set the $C$- and $DP$-lines equal to one another and find that $\a = \frac{5-d}{2}$. (Clearly, $\a \geq 2$ in this case, and so the condition on $F$ is more relaxed than the one on $C$.) When $d >1$, we set the $C$- and $F$-lines equal to one another and recognize $\a = 2$ as being optimal. We thus obtain the following conditions, which guarantee energy equality:
\begin{align}
\frac{2(3-d)}{p} + \frac{5-d}{q} & \le 3-d, \;\; 3 \le p,\;q\le p,  \; d \leq 1 \label{opt1} \\
\frac{2(3-d)}{p} + \frac{5-d}{q} &< 3-d, \;\; 3 \leq p < q,\;  d \leq 1 \label{opt2}\\
\frac{3-d}{p} + \frac{2}{q} &< \frac{3-d}{2}, \;\; 3 \leq p < q,\;  1<d < 3. \label{opt3}
\end{align}

Before proceeding, we make a few remarks concerning these conditions. First, we note that even though \eqref{DP} is valid only inside the square where $p\in [3,\infty)$ and $q\in [3,\infty)$, we can interpolate in order to include certain pairs $(p,q)$ outside of this region in \eqref{opt1}--\eqref{opt3}.  Second, in the case $d>1$, the optimal line drops below the Lions space $L^4 L^4$ in such a way that only the case $p<q$ yields new results; this line intersects the segment $[L^4 L^4, L^\infty L^3]$ at the space $L^\frac{6+2d}{3-d} L^\frac{6+2d}{1+d}$. Third, the inequality \eqref{opt2} once again becomes nonstrict in the case $d=0$. And finally, for the values $0\leq d \leq 1$, the point on the bisectrice separating the open and closed regions is  $L^{\frac{11-3d}{3-d}} L^{\frac{11-3d}{3-d}}$. When $d=1$, it becomes the classical Lions space $L^4 L^4$.  

Let us now address the case $p<3$. In this case, we use \eqref{DPenstr} to replace \eqref{DP} in the previous argument, while \eqref{C} and \eqref{F} are understood under the lighter restriction $p,q \geq 2$.  Also, note that the region under consideration now lies only in the cone $q>p$. The new $DP$-line pivots around the enstrophy point $L^2 L^6$ counterclockwise as $\a$ increases.  For $d\leq 1$, a non-trivial new region appears as $\a$ increases beyond $\a = \frac{5-d}{2}$. The $F$-line is less restrictive than $C$-line,  so we can disregard it. At $\a = \frac{5-d}{2}$, the $C$- and $DP$-lines intersect at  $L^{\frac{15 - 3d}{3-d}}L^3$. The point of intersection reaches its final state at the energy space $L^\infty L^2$ when $\a = 4$. In the process, it traverses the curve given by 
\begin{equation}
\label{xy_curve}
(18-6d)x^2 + (6-6d)xy - (21-7d)x - (7 - 3d)y + 6-2d = 0
\end{equation}
in coordinates  $x = p^{-1}$ and $y = q^{-1}$.  Notice that the curve in fact contains both $L^2 L^6$ and $L^\infty L^2$ for all values of $d$, as we expect it to. (Indeed, these points are the two axes of rotation for our lines.)  However, since we are restricted to the case when $p<3$, the part of the curve that we can use is limited to that connecting $L^{\frac{15 - 3d}{3-d}}L^3$ and $L^\infty L^2$.  The curve is a part of a hyperbola, as can be seen from the negative Hessian. (The exception is when $d=1$, in which case the curve is a parabola.)

For $d >1$, the two lines $C$ and $F$ coincide when $\a = 2$; at this value of $\a$, the new $DP$-line cuts through the $C$-line at space $L^\frac{6+6d}{3-d} L^\frac{6+6d}{1+3d}$, which is already inside the strip $p<3$.  It does not make sense to decrease $\a$ since doing so would move the $F$- and $DP$-lines clockwise inside the already discovered region. Increasing $\a$ above $2$ makes the $F$-line more relaxed, and the intersection point of $C$- and $DP$-lines falls on the same curve \eqref{xy_curve}.  This time, however, the curve begins farther to the right at the space $L^\frac{6+6d}{3-d} L^\frac{6+6d}{1+3d}$ and ends at $L^\infty L^2$, corresponding to the fixed range $2\leq \a \leq 4$. 

Finally, recall that in all the arguments above, we have assumed $q<\infty$ in order to ensure that the vanishing of the terms comes from the norm $L^q(I; L^p(\dom))$ and not from the Hausdorff measure of $S$. We also assumed $p<\infty$ in order ensure boundedness of the Riesz transforms on $L^p$. (This was necessary in order to bound the pressure term directly.) However, the cases $p = \infty$ and $q = \infty$ follow automatically by interpolation with the Leray--Hopf line, which lands the solution strictly inside the quadrant $q,p<\infty$.

Figures \ref{fig:dzero}-\ref{fig:d13} illustrate the new regions uncovered in each case.

\begin{figure}
\begin{minipage}{0.45\linewidth}
\centering
\includegraphics{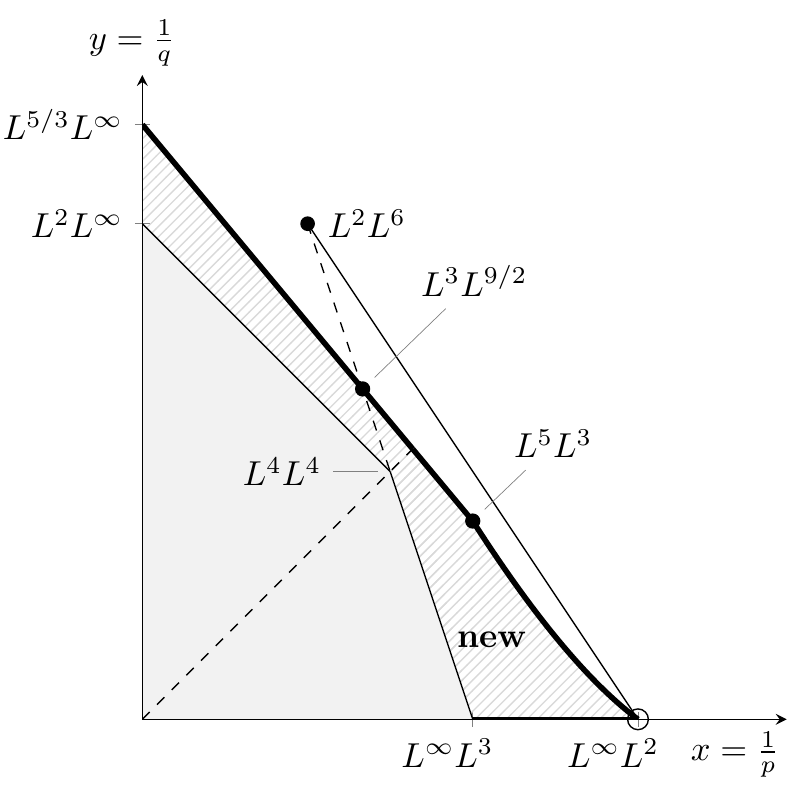}
\vspace{-9 mm} 
\caption{$d=0$.}
\label{fig:dzero}
\end{minipage}
\begin{minipage}{0.45\linewidth}
\centering
\includegraphics{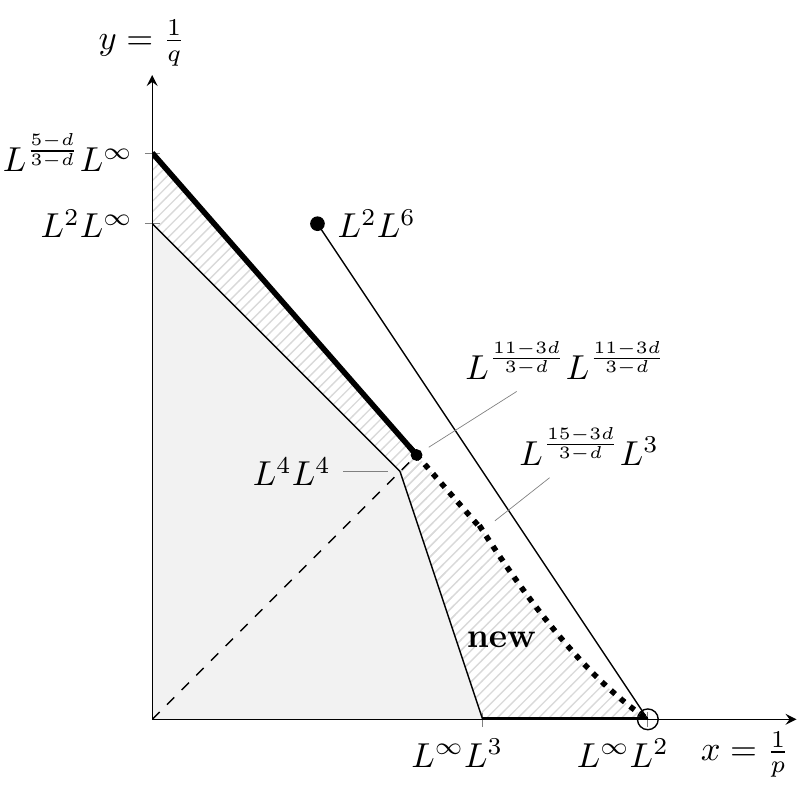}
\vspace{-9 mm} 
\caption{$0<d<1$.}
\label{fig:done}
\end{minipage}
\end{figure}

\begin{figure}
\begin{minipage}{0.45\linewidth}
\centering
\includegraphics{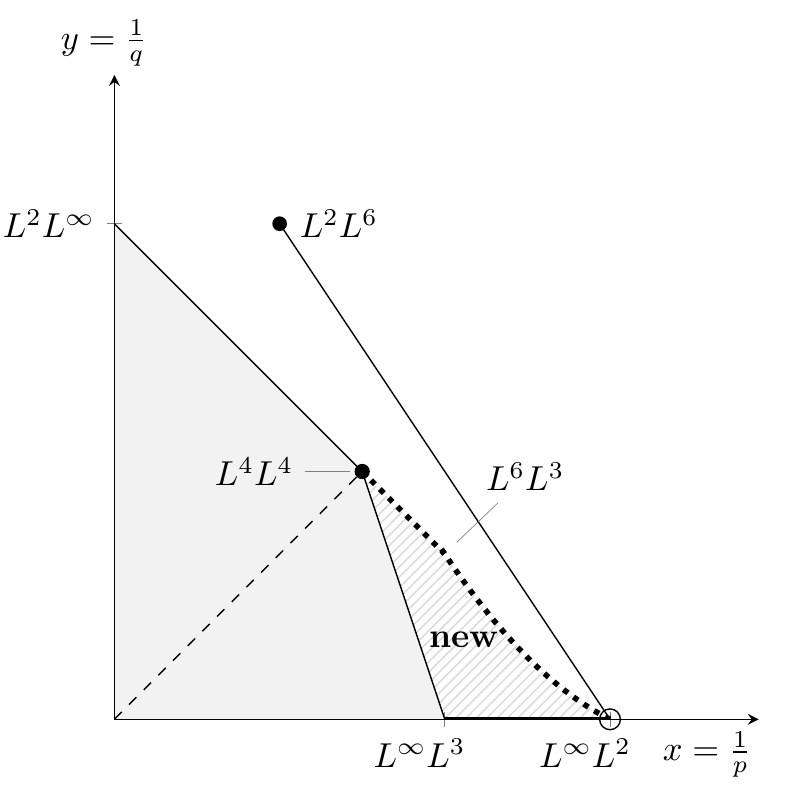}
\vspace{-9 mm} 
\caption{$d=1$.}
\label{fig:d1}
\end{minipage}
\begin{minipage}{0.45\linewidth}
\centering
\includegraphics{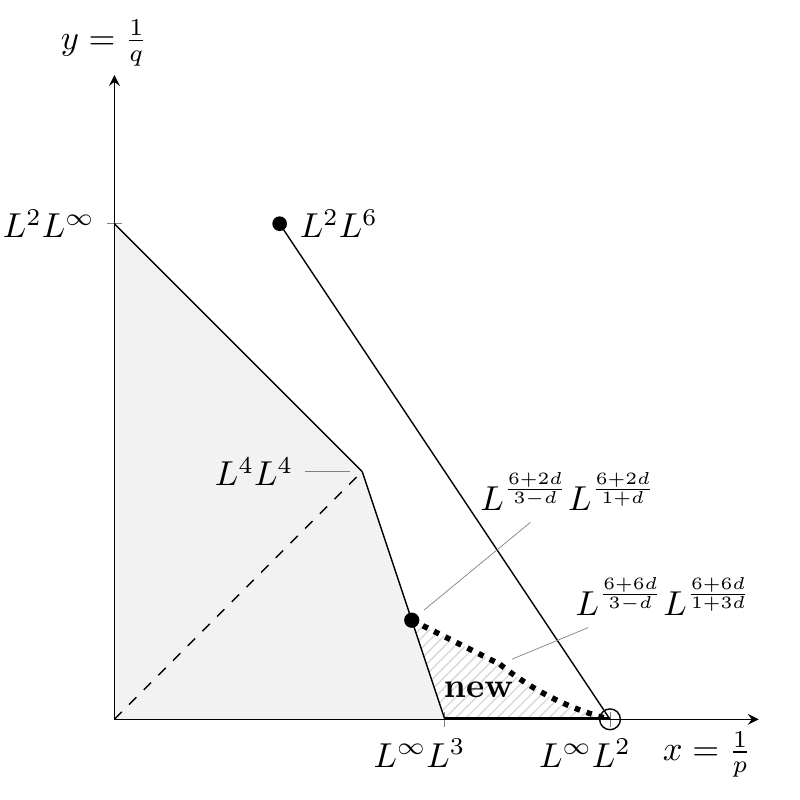}
\vspace{-9 mm} 
\caption{$1<d <  3$.}\label{fig:d13}
\end{minipage}
\end{figure}

\section{General singularities}\label{s:gen}
Even if the energy equality is known on each time interval of regularity including at the critical time, it is unknown whether energy equality holds globally on the time interval of existence of the weak solution. This is due to lack of a proper gluing procedure that could restore energy equality from pieces. In this section, we therefore address the question when singularity set $S$ is spread in space-time. In this case, we have no freedom in choosing the time scale of the covering cylinders; rather, the scale should already be built into the definition of the Hausdorff dimension. We choose to work with the classical parabolic dimension, i.e., $\a=2$ in our terms.

The main technical difference of this general case compared to the case of a one-time singularity is that when $s < 1$, the conclusion of Lemma \ref{l:conv} may not be valid.  Instead, we can only prove that the left side of \eqref{lemma_ineq} is bounded above by $H$ (multiplied by some constant which is independent of $\d$) under the stronger assumption $\s s + d\le \a =2$. This is achieved simply by bringing the exponent $s$ inside the sum.  However, the condition $\s s + d \le \a = 2$ is the sharpest possible under which the conclusion of the lemma holds, as one can see by considering an example of the opposite extreme, where all the intervals $I_i$ are disjoint.  However, the proof of the lemma in the case $s\ge 1$ does not depend on the intervals $I_i$ being nested; the proof and conclusion remain valid in this case.

Assume then that $S$ has finite $d$-dimensional parabolic Hausdorff measure for some $d\in[0, 1]$ but no other special properties.  (Our method does not yield anything new for $d>1$, so we do not treat these values of $d$.) Let $B_r(x)$, $\d$ be as above; then choose finitely many $(x_i, t_i)\in \R^3 \times (0,T]$ and $r_i\in (0,\d)$ such that $S\subset Q:=\bigcup_i Q_i$, where $Q_i = B_{r_i}(x_i)\times (t_i - r_i^2,\;t_i + r_i^2)$. Write $I_i = (t_i - 2r_i^2,\;t_i + 2r_i^2)$.  Let $Q^*$ denote the union of the double-dilated cylinders and $I = \bigcup_i I_i$.  Let $\psi$ be as above, and put $\phi_i = \psi(|x-x_i|/r_i) \psi(|t-t_i|/r_i^2)$ and $\phi = 1 - \sup_i \phi_i$.  

Let us note that in the special case $d=0$, $S$ is once again a finite point set.  The energy balance relation holds on each of the finitely many time-slices associated to each of the points in $S$ under the criteria of the previous section.  Therefore, it holds under these criteria for a general $0$-dimensional singularity set. Below we assume that $d\in (0,1]$.  

We also note that, as before, we have $|I|\to 0$ as $\d\to 0$, even though the intervals $I_i$ are no longer nested.  This is because
\begin{equation}
\label{e:Ito0}
|I|\le \sum_i |I_i|\lesssim \sum_i r_i^{d + (2-d)}< \d^{2-d} \sum_i r_i^d
\end{equation}
and because $d<2$ in all cases considered in this section.  

Assume $p\le q$.  Using bounds analogous to \eqref{e:C}, \eqref{e:F}, we see that $C,F\to 0$ whenever $\left(\frac{2p}{p-2} - 3\right)\frac{p-2}{p}\frac{q}{q-2} + d \le 2$, or, simplifying, 
\begin{equation}
\label{e:CFgen}
\frac3p + \frac{2-d}{q} \le \frac{3-d}{2} \quad (p\le q).
\end{equation}
Similarly, if $\infty>q\ge p\ge 3$, then $D,P\to 0$ whenever 
\begin{equation}
\label{e:DPgen_3pq}
\frac3p + \frac{2-d}{q} \le \frac{4-d}{3}
\quad (3\le p\le q<\infty).
\end{equation}
Of course, when $d\in [0,1]$, we have $\frac{4-d}{3}\le \frac{3-d}{2}$, so the restriction \eqref{e:DPgen_3pq} is limiting in this case.  

On the other hand, if $p<3$, then we use \eqref{e:DPpl3} and \eqref{e:defbeta}. Estimating 
\[
\|\n \phi\|_{L^\s L^\infty}^\s 
\le \sum_i \int r_i^{-\s} \chi_{I_i}(t)\,dt
\le \sum_i r_i^{2-\s},
\]
we see that $D,P\to 0$ whenever and $2-\s\ge d$, i.e.,
\begin{equation}
\label{e:DPgen_p3q}
\frac{4-d}{p} + \frac{2-d}{q} \le \frac{5-2d}{3}, \quad p<3.
\end{equation}
Notice that we could have also reached this inequality by interpolation.  This argument covers all terms under consideration in the case $p\le q$; it remains to deal with the case when $p>q$.  Most of the analysis from the single time-slice situation carries over in this case since Lemma \ref{l:conv} does not require nested $I_i$ in the case $s\ge 1$. However, the lack of freedom to choose $\a$ restricts the applicable range of pairs $(p,q)$.  After translating the condition $s(\s + d)\le 2$ into conditions on $C,D,P,F$, we see that $D,P$ are most stringent when $p\ge q\ge 3$ and correspond to the condition 
\[
\frac{3-d}{p} + \frac{2}{q} \le \frac{4-d}{3}, \quad 3\le q\le p.
\]
Using interpolation to treat the cases $p=\infty$, $q<3$, and $q = \infty$ as well, we can state our criteria for energy balance as follows:
\begin{subequations}
\label{e:gen_rest}
\begin{align}
\frac{2(3-d)}{p} + \frac{5-d}{q} \le 3-d, \quad q\le 3\le p  \\
\frac{3-d}{p} + \frac{2}{q} \le \frac{4-d}{3}, \quad 3\le q\le p \\
\frac{3}{p} + \frac{2-d}{q} \le \frac{4-d}{3}, \quad 3\le p\le q \\
\frac{4-d}{p} + \frac{2-d}{q} \le \frac{5-2d}{3}, \quad p\le 3\le q.
\end{align}
\end{subequations}
As $d\to 1^-$, these criteria collectively collapse to the region implicated by the Lions $L^4 L^4$ condition.  However, when $d\in (0,1)$ we obtain a new region bounded by the points $L^{\frac{5-d}{3-d}}L^\infty$,  $L^3 L^{\frac{9-3d}{2-d}}$, $L^{\frac{15 - 3d}{4-d}} L^{\frac{15 - 3d}{4-d}}$, $L^{\frac{6-3d}{1-d}}L^3$, $L^\infty L^{\frac{5-2d}{12-3d}}$.  See Figures \ref{fig:0<d<1_gen} and \ref{fig:d=1_gen}.

\begin{figure}
\begin{minipage}{0.45\linewidth}
\includegraphics{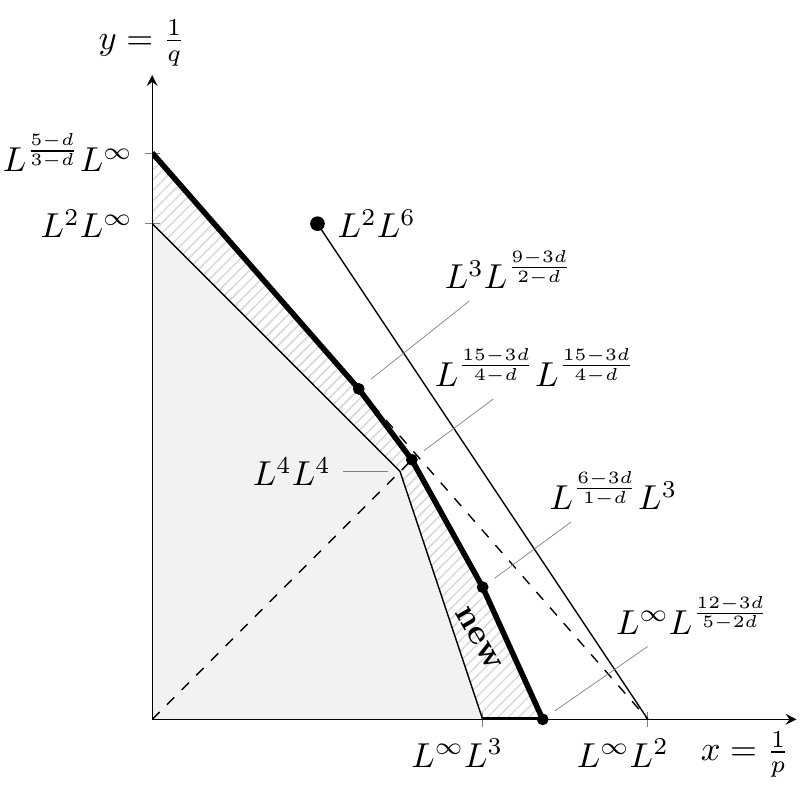}
\vspace{-8 mm}
\caption{$0<d<1$.}\label{fig:0<d<1_gen}
\end{minipage}
\begin{minipage}{0.45\linewidth}
\includegraphics{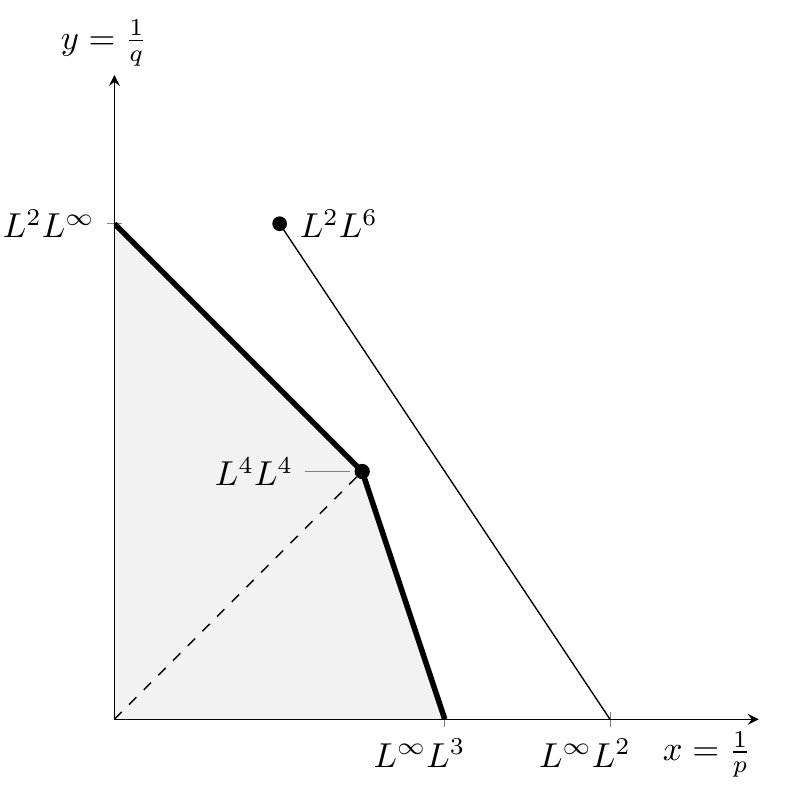}
\vspace{-8 mm}
\caption{$d=1$.}\label{fig:d=1_gen}
\end{minipage}
\end{figure}

\section{Fractional NSE}\label{s:frac}
In this section, we present extensions of the results for the classical NSE to the case of fractional dissipation $\g<1$:
\begin{equation}
\label{e:momentum_g}
\p_t u + u\cdot \n u + \nu \L_{2\g} u = - \n p
\end{equation}
\begin{equation}
\label{e:divfree_g}
\n \cdot u = 0
\end{equation}
where $\widehat{\L_s u} = |\xi|^s \widehat{u}$. We define the (Onsager) regular and singular sets as in the classical case.  We also define the postinitial singularity set $S$ as before.  In the fractional dissipation case,  weak solutions belong to $L^2 H^\g \cap L^\infty L^2$, and the energy equality can be written 
\begin{equation}
\label{localeefrac}
\begin{split}
\int_{\R^3} |u(t)|^2 \phi & - \int_{\R^3} |u(s)|^2 \phi - \int_{\R^3 \times (s,t)} |u|^2 \p_t \phi \\
&= \int_{\R^3 \times (s,t)} |u|^2  u \cdot \n \phi + 2 \int_{\R^3 \times (s,t)} p\ u \cdot \n \phi - 2\nu \int_{\R^3 \times (s,t)} |\L_{\g} u|^2 \phi \\ 
&-2\nu \int_{\R^3 \times (s,t)} \L_{\g} u \cdot u\L_{\g}\phi - 2\nu \int_{\R^3 \times (s,t)} \L_{\g} u \cdot [\L_{\g}(u\phi) - (\L_{\g}u)\phi - u\L_{\g}\phi].
\end{split}
\end{equation}
As in the classical case, this equality is valid for $\phi\in W^{1,\infty}(\R^3\times [0,T])$ which are supported outside $S$. We label our terms in the same manner as in the classical case:
\[
A - B - C = D + 2P - 2\nu E - 2\nu F - 2\nu G.
\]
As before, convergence of $A,B,E$ is obvious; proving energy equality amounts to showing that the other terms vanish.  

For sufficiently regular $f$ and $\g\in (0,2)$, we have 
\[
\L_\g f(x) = -c_\g \int \frac{\d_{-z} \d_z f(x)}{|z|^{3+\g}}\,dz = \widetilde{c}_\g \ p.v. \int \frac{\d_z f(x)}{|z|^{3 + \g}}\,dz,
\]
where $\d_z$ denotes the difference operator $\d_z f(x) = f(x + z) - f(x)$. 

\begin{lemma}
Suppose $\phi\in W^{1,a}$ for some $a\in [1,\infty]$, and let $\g\in (0,1)$.  Then $\L_\g \phi\in L^a$, and we have the bound
\begin{equation}
\label{Lgphi_bd}
\|\L_\g \phi \|_{L^a} \lesssim \|\phi\|_{L^a}^{1-\g} \|\n \phi \|_{L^a}^\g.
\end{equation}
\end{lemma}
\begin{proof}
For any $r>0$, we estimate 
\begin{align*}
\|\L_\g \phi \|_{L^a}
& =  \left\| \int_{|z|\le r} \frac{ \d_z \phi }{|z|^{3 + \g}}\,dz + \int_{|z|> r} \frac{ \d_z \phi }{|z|^{3 + \g}}\,dz \right\|_{L^a} \\
& \le \int_{|z|\le r} \frac{\| \n \phi \|_{L^a}}{|z|^{2 + \g}}\,dz + \int_{|z|>r} \frac{2 \| \phi \|_{L^a}}{|z|^{3 + \g}}\,dz \\
& \le r^{-\g} [ r \|\n \phi\|_{L^a} + 2\|\phi\|_{L^a}].
\end{align*}
We put $r = \|\phi\|_{L^a} \|\n \phi\|_{L^a}^{-1}$ to optimize. The bound \eqref{Lgphi_bd} follows immediately.
\end{proof}

\begin{lemma}\label{p:commut}
Let $u\in H^\g \cap L^p$, $p>2$, $\g\in (0,1)$, and $\phi\in W^{1,\frac{2p}{p-2}}$.  Then 
\begin{equation}
\| \L_\g (u\phi) - (\L_\g u)\phi - u\L_\g \phi\|_{L^2} \lesssim \|u\|_{L^p}\|\phi\|_{L^{\frac{2p}{p-2}}}^{1-\g} \|\n \phi \|_{L^{\frac{2p}{p-2}}}^\g.
\end{equation}
The inequality continues to hold when $p=2$ and $2p/(p-2)$ is replaced by $\infty$.
\end{lemma}
\begin{proof}
We use the identity
\[
[ \L_{\g}(u\phi) - (\L_{\g}u)\phi - u\L_{\g}\phi](x)
 = - \int \frac{\d_z u(x)\, \d_z \phi(x)}{|z|^{3 + \g}}\,dz
\]
and estimate the right side of this equality.  Let $r>0$ be arbitrary for now.  Then 
\begin{align*}
\left\| \int \frac{\d_z u\, \d_z \phi}{|z|^{3 + \g}}\,dz \right\|_{L^2}
& \le \int_{|z|\le r} \frac{ \|\d_z u\, \d_z \phi \|_{L^2}}{|z|^{3 + \g}}\,dz + \int_{|z|>r} \frac{ \|\d_z u\, \d_z \phi \|_{L^2}}{|z|^{3 + \g}}\,dz\\
& \le 2 \|u\|_{L^p} \left[ \int_{|z|\le r} \frac{ \|\n \phi \|_{L^{\frac{2p}{p-2}}}}{|z|^{2 + \g}}\,dz + \int_{|z|>r} \frac{ 2 \| \phi \|_{L^{\frac{2p}{p-2}}}}{|z|^{3 + \g}}\,dz \right] \\
& \le 2r^{-\g} \|u\|_{L^p} [r \|\n\phi \|_{L^{\frac{2p}{p-2}}} + 2 \|\phi\|_{L^{\frac{2p}{p-2}}}].
\end{align*}
Put $r = \|\phi\|_{L^{\frac{2p}{p-2}}} \|\n \phi\|_{L^{\frac{2p}{p-2}}}^{-1}$ to complete the proof.
\end{proof}

We combine the two propositions above and apply them to our original test function $\phi$: 
\begin{align*}
& \hspace{-10 mm} \int \|u\L_{\g}\phi\|_{L^2}^2 + \|\L_{\g}(u\phi) - (\L_{\g}u)\phi - u\L_{\g}\phi\|_{L^2}^2 \,dt \\
& \lesssim \int (\|u\|_{L^p} \|\phi\|_{L^{\frac{2p}{p-2}}}^{1-\g} \|\n \phi\|_{L^{\frac{2p}{p-2}}}^{\g})^2\,dt \\
& \le \|u\|_{L^q L^p}^2 \|\phi\|_{L^\infty L^{\frac{2p}{p-2}}}^{2(1-\g)} \|\n \phi\|_{ L^{\frac{2q\g}{q-2}} L^{\frac{2p}{p-2}}}^{2\g}.
\end{align*}
Now with the bound $|\n \phi(x,t)|\le \sup_i |\n \phi_i(x,t)|$, we obtain 
\begin{equation}
\label{FG_frac}
\|\n \phi\|_{ L^{\frac{2q\g}{q-2}} L^{\frac{2p}{p-2}}}^{\frac{2q\g}{q-2}}
\le \int \left( \sum_i r_i^{-\frac{2p}{p-2} + 3} \chi_{I_i}(t) \right)^{\frac{p-2}{p}\frac{q}{q-2}\g}\,dt.
\end{equation}
So we can use Lemma \ref{l:conv} to give conditions on when $|F|+|G|\to 0$, depending on whether we are dealing with the one-slice or general type singularity.

\subsection{One-time singularity case, $\frac12<\g<1$}
We recall some of the conditions for the vanishing of $C$ and $D+P$ and (using the lemma) add to them conditions for the vanishing of $F+G$.  Note that the restriction \eqref{DPfracslice} below on $D+P$ is only valid inside the square $p,q\ge 3$, just as before.  We deal with this case first and investigate the case $p<3$ separately:
\begin{equation}
\label{Cfracslice}
\frac{3-d}{p} + \frac{\a}{q}\le \frac{3-d}{2},\; p\ge q \geq 2;\quad 
\frac{3-d}{p} + \frac{\a}{q}< \frac{3-d}{2},\; 2\leq p< q
\end{equation}
\begin{equation}
\label{DPfracslice}
\frac{3-d}{p} + \frac{\a}{q}\le \frac{2 + \a - d}{3},\; p\ge q \geq 3;\quad 
\frac{3-d}{p} + \frac{\a}{q}< \frac{2 + \a - d}{3},\; 3\leq p< q
\end{equation}
\begin{subequations} \label{FG}
\begin{align}
\label{FGfracslicetop}
\frac{(3-d)\g}{p} + \frac{\a}{q} & \le \frac{(3-d)\g + \a - 2\g}{2}, \quad \frac1q - \frac{\g}{p} \ge \frac{1-\g}{2}, \;\; p,q\ge 2 \\
\label{FGfracslicebottom}
\frac{(3-d)\g}{p} + \frac{\a}{q} & < \frac{(3-d)\g + \a - 2\g}{2}, \quad \frac1q - \frac{\g}{p} < \frac{1-\g}{2}, \;\; p,q\ge 2.
\end{align}
\end{subequations}

The line $\frac1q - \frac{\g}{p} = \frac{1-\g}{2}$ joins $L^{\frac{2}{1-\g}} L^\infty$ with $L^2 L^2$.  It plays the role for $F+G$ that the bisectrice plays for $C$ and $D+P$. Also note that for each restriction, all inequalities are nonstrict in the special case $d=0$, just as before.  

When $d\le 5-4\g$, we find using the same argument as in the classical case that $\a=\frac{5-d}{2}$ gives the optimal region.  At this value of $\a$, $\eqref{Cfracslice}$ and $\eqref{DPfracslice}$ coincide, and \eqref{FGfracslicetop} and \eqref{FGfracslicebottom} are less restrictive than \eqref{Cfracslice} and \eqref{DPfracslice}.  Furthermore, since the line corresponding to \eqref{Cfracslice} rotates about $L^\infty L^2$, we may use interpolation to remove the restriction $q\ge 3$.  

In the case $p<3$, we repeat the argument used for the classical NSE and make changes where necessary. Assume first that $\g\ge \frac34$. Then we have 
\begin{equation}
\label{frac_D}
|D|+|P|\le \|u\|_{L^2 H^1}^{3\b} \|u\|_{L^q L^p}^{3(1-\b)}\|\n \phi\|_{L^\s L^\infty}, 
\end{equation}
where 
\begin{equation}
\frac13 = \frac{(3-2\g)\b}{6} + \frac{1-\b}{p}\implies 
\b = \frac{6-2p}{6-(3-2\g)p};\;
1-\b = \frac{(2\g-1)p}{6-(3-2\g)p}
\end{equation}
\begin{equation}
\frac1\s = 1 - \frac{3\b}{2} - \frac{3(1-\b)}{q} = \frac{2\g pq - 3p(2\g-1) - 3q}{(6-(3-2\g)p)q}.
\end{equation}
Now 
\[
\|\n \phi\|_{L^\s L^\infty}^\s = \int \sup_i r_i^{-\s} \chi_{I_i}(t)\,dt \le \int \sup_j 2^{j\s} \chi_{J_j}(t)\,dt \lesssim \sum_j (2^{\a-\s})^{-j},
\]
and the sum on the right is bounded whenever $\s < \a$.  Substituting in for $\s$ and simplifying, we obtain
\begin{equation}
\label{frac_Drest_pl3}
\frac{2 + \a}{p} + \frac{(2\g-1)\a}{q} < \frac{3-2\g + 2\a\g}{3}.
\end{equation}

Note that as $\a$ increases, the line corresponding to equality rotates counterclockwise about $L^2 L^{\frac{6}{3-2\g}}$.  Combining \eqref{Cfracslice} and \eqref{frac_Drest_pl3} with inequality replaced by equality in both cases, we find the curve
\begin{equation}\label{frac_xy_curve}
\begin{split}
6(3-d)x^2 + 6(6\g - 5 - (& 2\g - 1)d)xy - (4\g +3)(3-d)x \\&  - (22\g - 15 - (2\g-1)3d)y + 2\g(3-d) = 0,
\end{split}
\end{equation}
where $x = p^{-1}$ and $y = q^{-1}$.  Notice that the curve contains both $L^2 L^{\frac{6}{3-2\g}}$ and $L^\infty L^2$, as we expect it to. (Indeed, these points are the two axes of rotation for our lines.)  However, since we are restricted to the case when $p<3$, the part of the curve that we can use is limited to that connecting $L^{\frac{15 - 3d}{3-d}}L^3$ and $L^\infty L^2$.  

If $\g<\frac34$, then \eqref{frac_D} is not valid for all values of $p,q$.  In particular, we need $3\b \le 2$ for the obvious application of H\"older to be valid, which translates to $p\ge \frac{3}{2\g}$.  When $\g\in (\frac12, \frac34)$, we also see that the two places where the curve \eqref{frac_xy_curve} crosses the $x$-axis are at $x=\frac12$ and $x= \frac{2\g}{3}$.  When $\g\in (\frac12, \frac34)$, we have $\frac{2\g}{3}\in (\frac13, \frac12)$.  So the curve still gives us a meaningful restriction up to the point where it crosses the $x$-axis for the first time.  Once $\g<\frac12$, however, we have $\frac{2\g}{3}<\frac13$, so the use of enstrophy does not allow us to make any statement about the range $p<3$.  

All in all, our criteria for energy equality in the case $\frac12 < \g < 1$, $0\le d\le 5-4\g$ can be stated as 
\begin{equation}
\frac{2(3-d)}{p} + \frac{5-d}{q}\le 3-d,\; p\ge q;\quad 
\frac{2(3-d)}{p} + \frac{5-d}{q}< 3-d,\; 3\leq p< q
\end{equation}
\begin{equation}
\begin{split}
6(3-d)x^2 + 6(& 6\g - 5 - ( 2\g - 1)d)xy - (4\g +3)(3-d)x \\&  - (22\g - 15 - (2\g-1)3d)y + 2\g(3-d) > 0,
\quad \frac13 < x< \min\{ \frac12, \frac{2\g}{3}\}.
\end{split}
\end{equation}
Once again, strict inequalities are replaced by nonstrict ones if $d=0$.

\begin{figure}
\begin{minipage}{0.45\linewidth}
\includegraphics{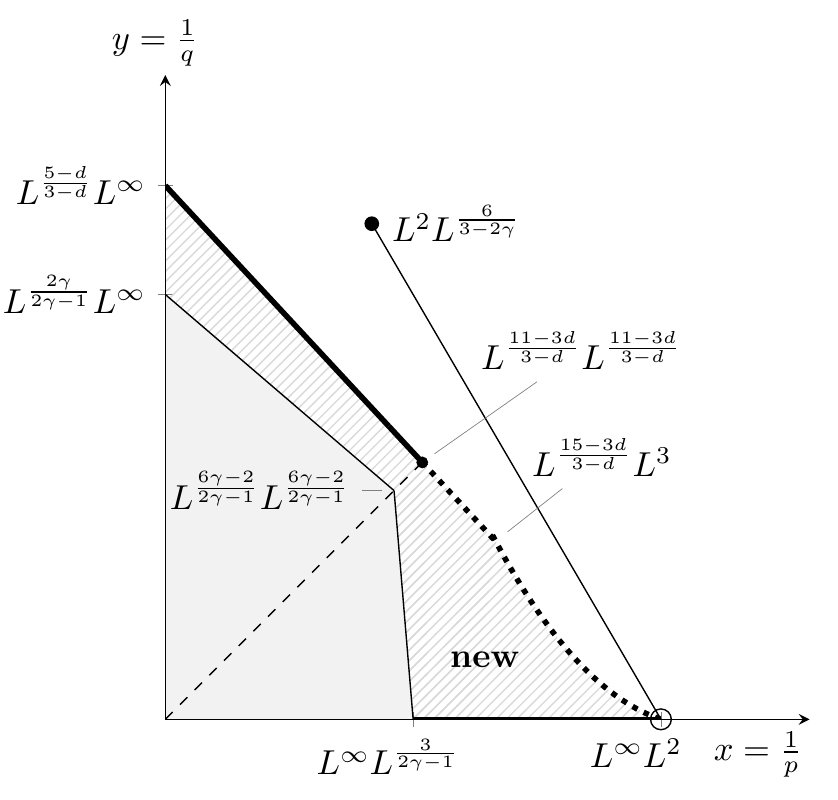}
\vspace{-8 mm}
\caption{$\frac34\le\g<1$, $0<d\leq 5-4\g$.}
\label{fig:frac_78}
\end{minipage}
\begin{minipage}{0.45\linewidth}
\includegraphics{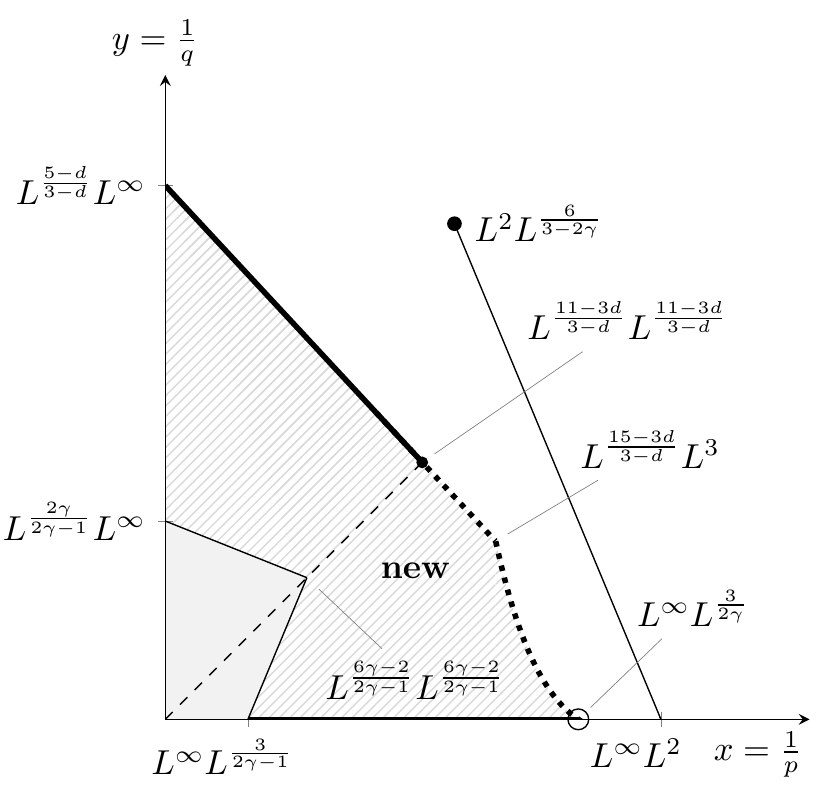}
\vspace{-8 mm}
\caption{$\frac12<\g<\frac34$, $0<d\leq 5-4\g$.}
\label{fig:frac_58}
\end{minipage}
\end{figure}
Figures \ref{fig:frac_78} and \ref{fig:frac_58} diagram  our results for a fixed value of $d\in (0, 5-4\g)$ (we use $d=\frac23$) and varying $\g\in (\frac12, 1)$.  Note that $L^{\frac{6\g-2}{2\g-1}} L^{\frac{6\g-2}{2\g-1}}$ serves as the analogue of the Lions space in the present context, because interpolation between this space and $L^2 H^\g$ lands in the Onsager space $L^3 B^{1/3}_{3,c_0}$.

As we take $d\to 5-4\g$ from below, the new region above the bisectrice collapses to the segment $[L^\infty L^{\frac{2\g}{2\g-1}}, L^{\frac{6\g-2}{2\g-1}}L^{\frac{6\g-2}{2\g-1}}]$. In this respect, the value $d=5-4\g$ serves a similar role to the value $d=1$ in the classical case.  Things are slightly more complicated when $5-4\g<d<3$. In this case, setting $\a$ equal to its usually optimal value of $\a = \frac{5-d}{2}$ places too heavy a burden on $F+G$; for a fixed $p$, we must increase $\a$ to optimize until the restrictions on $C$ and $F+G$ coincide.  An elementary computation gives the optimal value of $\a$ to be
\begin{equation}
\label{eq:optalphafracdlarge}
\a_{CF}(x) = (3-d)(1-\g)(1-2x) + 2\g \quad \quad (x = p^{-1}).
\end{equation}
We see then that as $x$ increases, the optimal value of $\a$ decreases.  When $p\ge 3$, the restriction on $D+P$ is always less stringent for this value of $\a$ than the corresponding restriction for $C$.  However, as $x$ increases beyond $\frac13$, $\a_{CF}(x)$ eventually becomes sufficiently small so that \eqref{frac_Drest_pl3} becomes limiting once again.  At this point, the optimal restriction is once again determined by the intersection of the $C$ and $D+P$ lines, following the curve \eqref{frac_xy_curve}.  Indeed, along the curve \eqref{frac_xy_curve}, $\a$ is given by 
\begin{equation}
\a_{CDP}(x) = \frac{3[(3-d)(2\g - 1) - 2](1-2x) + 4\g}{2(2\g - 3x)}.
\end{equation}
Now
\[
\a_{CDP}(1/3) = \frac{5-d}{2} < 2\g < \a_{CF}(1/3),
\]
whereas 
\[
\a_{CDP}(1/2) = \frac{4\g}{4\g-3}>2\g = \a_{CF}(1/2) \quad (3/4<\g<1),
\]
\[
\lim_{x\to \frac{2\g}{3}^-}\a_{CDP}(x) = \infty > \a_{CF}(2\g/3) \quad (1/2<\g\le 3/4).
\]
So there must be some $x_0\in (\frac13, \min\{\frac12, \frac{2\g}{3}\})$, where $\a_{CDP}(x_0) = \a_{CF}(x_0)$.  The actual value of $x_0$ does not seem to take a particularly enlightening form in general, but it can be easily calculated given $\g\in (\frac12, 1)$ and $d\in (5-4\g, 3)$.  See Figures \ref{fig:frac_78_dlarge} and \ref{fig:frac_58_dlarge}. 

\begin{figure}
\begin{minipage}{0.40\linewidth}
\includegraphics{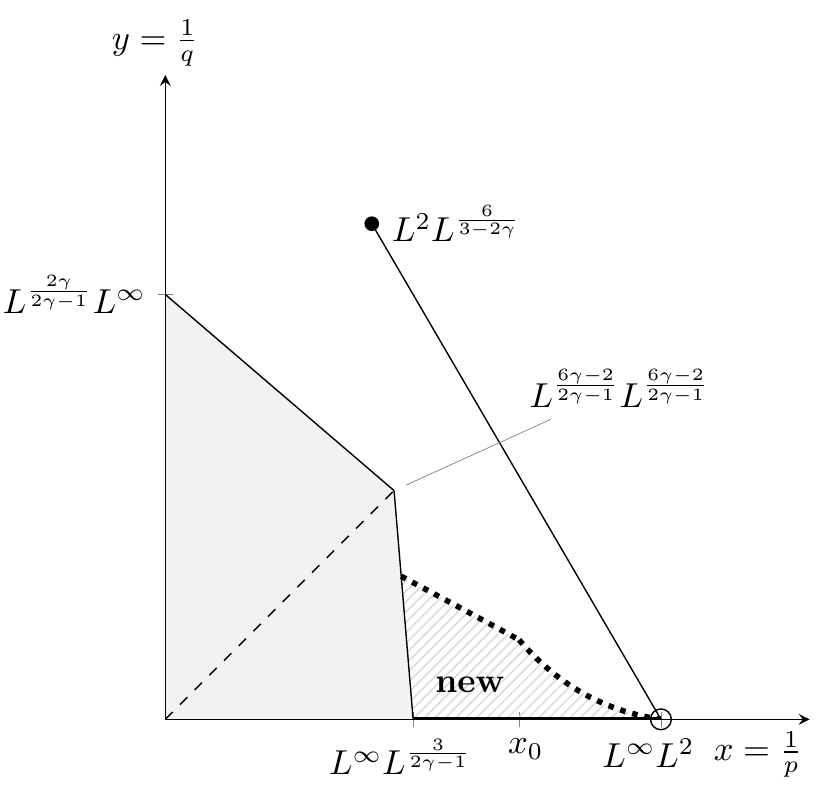}
\vspace{-8 mm}
\caption{$\frac34\le \g<1$, $\;5-4\g<d<3$.}
\label{fig:frac_78_dlarge}
\end{minipage}
\hspace{8 mm}
\begin{minipage}{0.4\linewidth}
\includegraphics{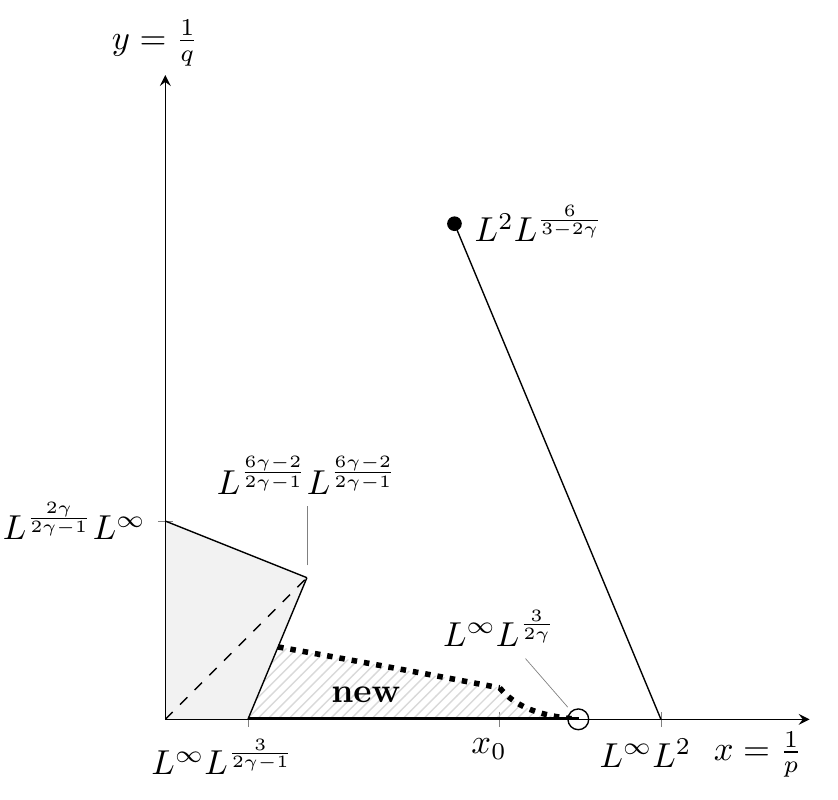}
\vspace{-8 mm}
\caption{$\frac12<\g<\frac34$, $\;5-4\g<d<3$.}
\label{fig:frac_58_dlarge}
\end{minipage}
\end{figure}

Altogether, the criteria for energy equality in the case $\g\in (\frac12, 1)$ and $d\in (5-4\g, 3)$ can be stated as 
\begin{equation}
4(1-\g)(3-d)xy - 2(3-d + (d-1)\g)y + (1-2x)(3-d) > 0, \quad x<x_0
\end{equation}
\begin{equation}
\begin{split}
6(3-d)x^2 + 6(& 6\g - 5 - ( 2\g - 1)d)xy - (4\g +3)(3-d)x \\&  - (22\g - 15 - (2\g-1)3d)y + 2\g(3-d) > 0,
\quad x_0 < x< \min\{ \frac12, \frac{2\g}{3}\}.
\end{split}
\end{equation}

\subsection{One-time singularity  case, $0<\g\le\frac12$} Much of the analysis of the previous subsection carries over to the case when $\g\in (0,\frac12]$.  However, there are a few important differences.  For one thing, the Lions region is the single point $L^\infty L^\infty$ when $\g=\frac12$ and trivial otherwise.  Second, the case $d>5-4\g$ is geometrically impossible since $5-4\g>3$ here.  Finally, we cannot say anything about the region $p<3$.  As was mentioned earlier, the enstrophy argument used to deal with this region for larger values of $\g$ does not apply when $\g\in (0,\frac12)$. In fact, we cannot even get any new information by interpolation with the Leray--Hopf line since the point $L^2 L^{\frac{6}{3-2\g}}$ lies on the line $x=\frac13$ when $\g = \frac12$ and to the right of this line when $\g<\frac12$.  So the region for which we have proved energy equality is independent of $\g$ for $\g<\frac12$; the region depends only on $d$.  See Figure \ref{fig:frac_14}.   \color{black}

\begin{figure}
\includegraphics{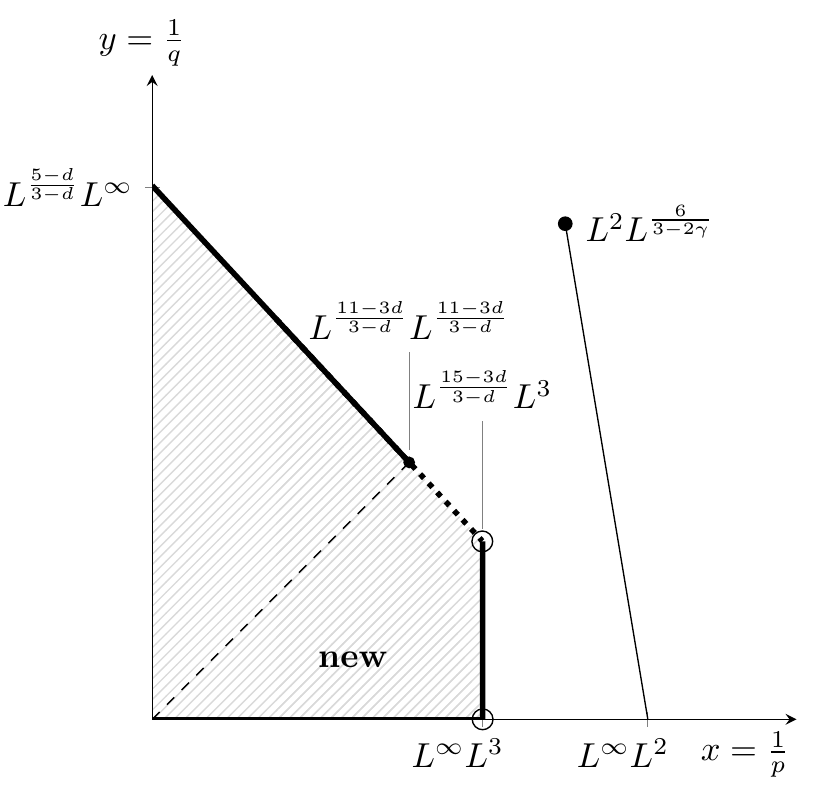}
\caption{$0<\g\leq \frac12$, $d<3$.}
\label{fig:frac_14}
\end{figure}

\subsection{General singularities}
We fix $\a = 2\g$ in consideration of the natural scaling.  The restrictions corresponding to $C$, $D+P$, and $F+G$ become

\begin{equation}
\label{Cfrac}
\frac{3-d}{p} + \frac{2\g}{q}\le \frac{3-d}{2},\; p\ge q \geq 2;\quad 
\frac{3}{p} + \frac{2\g - d}{q}\le \frac{3-d}{2},\; 2\leq p< q
\end{equation}
\begin{equation}
\label{DPfrac}
\frac{3-d}{p} + \frac{2\g}{q}\le \frac{2 + 2\g - d}{3},\; p\ge q \geq 3;\quad 
\frac{3}{p} + \frac{2\g - d}{q}\le \frac{2 + 2\g - d}{3},\; 3\leq p< q
\end{equation}
\begin{subequations}
\begin{align}
\label{FGfractop}
\frac{3-d}{p} + \frac{2}{q} & \le \frac{3-d}{2}, \quad \frac1q - \frac{\g}{p} \ge \frac{1-\g}{2}, \;\; p,q\ge 2 \\
\label{FGfracbottom}
\frac{3\g}{p} + \frac{2\g-d}{q} & \le \frac{3\g-d}{2}, \quad \frac1q - \frac{\g}{p} < \frac{1-\g}{2}, \;\; p,q\ge 2.
\end{align}
\end{subequations}

We will not present figures pertaining to this particular situation, as the reader can easily verify conditions above for any particular values of $\g,d,p,q$. However, we make several comments. 

First, we note that the measure of $I$ may not vanish for certain combinations of $\g, d$.  Mimicking the argument of \eqref{e:Ito0} only gives $|I|\to 0$ when $d\le 2\g$.  If $d>2\g$, then we continue with the additional assumption that $\cH_d(S)$ is actually zero (rather than merely finite, as we usually assume).

Assume first that $\g\in (\frac12,1)$. Then \eqref{DPfrac} is more stringent than \eqref{Cfrac} when $d<5-4\g$; the two inequalities coincide when $d=5-4\g$.  At this value of $d$, the region satisfying \eqref{Cfrac}, \eqref{DPfrac} is exactly the region already covered by the analogue of the Lions result.  So only the case $d<5-4\g$ can give new information.  However, in contrast to the classical case, the restrictions \eqref{FGfractop}, \eqref{FGfracbottom} are not always superfluous.  If $\g< \frac12$, then the Lions region is  trivial, and consequently the value $d = 5 - 4\g$ has no special significance for our argument in the case of a general $2\g$-parabolic $d$-dimensional singularity with $\g\in (0,\frac12)$. 

When $d=0$, the singularity set can be covered by finitely many time-slices, and the region covered is the same as in the one-slice case.  When $d\in (0, 2\g - 1)$, the $DP$-lines are limiting, but there is still a nontrivial region covered in the range $p<3$ by interpolation.  This region disappears when $d = 2\g - 1$, but the $DP$-lines remain the limiting restriction until $d$ surpasses the value $\frac12(5+\g - \sqrt{ 9\g^2 - 18\g + 25})$, at which point the lower $FG$-line (corresponding to \eqref{FGfracbottom}) cuts into both the upper and the lower $DP$-lines.  This situation prevails until $d$ reaches the value $\frac{5\g-4\g^2}{3-2\g}$, at which point the lower $FG$-line becomes more stringent than the lower $DP$-line everywhere below the bisectrice.  However, at this point, the upper $FG$-line is still less stringent than the upper $DP$-line; this changes once $d$ surpasses $1$. Note that the point $L^{\frac{5-d}{3-d}}L^\infty$ is no longer included in the region covered for $d>1$.  Rather, the upper $FG$-line lies strictly below the interpolation line obtained in the region $q<3$ from the uppermost point on the $DP$ segment.  When $d$ lies in the range $d\in  [1,  2\g + 1 - 2\sqrt{3\g^2 - 3\g + 1})$, the upper $DP$-line remains more stringent than the $FG$-lines on a small segment.  However, once $d\ge 2\g + 1 - 2\sqrt{3\g^2 - 3\g + 1}$, the $FG$ restrictions are limiting in all cases.

There are a few larger values of significance for $d$, but they involve the interaction between the $FG$-lines and the Lions region rather than the $FG$-lines and the other restrictions imposed by our method.  We describe briefly the bifurcations of the diagrams.  When $d$ reaches the value $2-\g$, the Lions point $L^{\frac{6\g - 2}{2\g - 1}}L^{\frac{6\g - 2}{2\g - 1}}$ lies on the lower $FG$ segment.  When $d= \g(5-4\g)$, the new region below the bisectrice disappears entirely (since $\frac{3\g - d}{6\g} = \frac{2\g - 1}{3}$ for this value of $d$).  The new region disappears entirely into the Lions region once $d=\frac{2-\g}{\g}$.  Indeed, at this value of $d$, we have $\frac{3-d}{4} = \frac{2\g-1}{2\g}$; furthermore, both the upper $FG$-line and the line containing the upper part of the boundary for the Lions region pass through $L^\infty L^2$.  Therefore, the upper $FG$-line collapses to (a portion of) the boundary of the Lions region when $d = \frac{2-\g}{\g}$.  

\begin{remark} Finally, we make a remark about the case $\g>1$. The main technical reason why this case eludes our analysis is a failure to produce a proper cutoff function $\f$ for which $\L_\g \f$ would remain under control, as $\n_x \f$ would already develop jump discontinuities. However if $d=0$, i.e., a finite-point set $S$, one can construct each $\f$ from $\f_i$'s having disjoint support, allowing the analysis to be carried out. In this case, the region of conditions is the same as what is shown in Figure~\ref{fig:dzero}, except that the equation for the hyperbola connecting $L^5 L^3$ to the energy space is now given by \eqref{frac_xy_curve} with $d=0$:
\begin{equation}
18 x^2 + (6\g - 5)6xy - (12\g +9) x - (22\g - 15)y + 6\g = 0, \quad \frac13 \le x \le \frac12.
\end{equation}
\end{remark}

%


\def\cprime{$'$}

\end{document}